\documentclass[leqno,12pt]{article}

\usepackage{epsfig}

\usepackage{amsmath}
\usepackage{amscd}
\usepackage{amsopn}
\usepackage{amsthm}
\usepackage{amsfonts,amssymb}
\usepackage{amsfonts,bbm}
\usepackage{latexsym}

\makeatletter
\@addtoreset{equation}{section}
\makeatother

\newtheorem{lemma}{LEMMA}[section]
\newtheorem{proposition}[lemma]{PROPOSITION}
\newtheorem{corollary}[lemma]{COROLLARY}
\newtheorem{theorem}[lemma]{THEOREM}
\newtheorem{remark}[lemma]{REMARK}

\newtheorem{example}[lemma]{EXAMPLE}

\newtheorem{definition}[lemma]{DEFINITION}

\newcommand{\real}{\mathbbm{R}}
\newcommand{\nat}{\mathbbm{N}}

\renewcommand{\a}{\alpha}

\newcommand{\g}{\gamma}

\newcommand{\vp}{\varphi}
\newcommand{\ve}{\varepsilon}

\newcommand{\reald}{{\real^d}}

\newcommand{\on}{\quad\text{ on }}
\newcommand{\und}{\quad\mbox{ and }\quad}
\newcommand{\inv}{^{-1}}
\newcommand{\ov}{\overline}

\newcommand{\V}{\mathcal V}  
\newcommand{\W}{\mathcal W}  
\newcommand{\C}{\mathcal C}  

\newcommand{\F}{\mathcal F}

\renewcommand{\H}{{\mathcal H}}
\newcommand{\B}{\mathcal B}
 \newcommand{\Su}{\mathcal S}
\newcommand{\M}{\mathcal M}

\newcommand{\Q}{\mathcal Q}
\newcommand{\U}{{\mathcal U}}

\newcommand{\itemframe}%
{\setlength{\parskip}{10pt}\begin{enumerate} \setlength{\topsep}{10pt}%
\setlength{\itemsep}{15pt}\setlength{\parsep}{5pt}}

\newcommand{\vx}{\ve_x}

\newcommand{\cc}{\mathfrak C}

\title{Darning and gluing of  diffusions}

\author{WOLFHARD HANSEN} 

\date{}   

\begin{document}

\maketitle

\begin{abstract}
We  study darning  of   compact sets (darning and gluing of finite unions of compact sets),
which are not thin at any of their points, in a~potential-theoretic framework which may be described,
analytically,  in terms of harmonic kernels/harmonic functions or,  probabilistically, in terms of a~diffusion.

This is accomplished without leaving our kind of setting so that the procedure can be iterated without any problem. 
It applies to darning and gluing of compacts in Euclidean spaces (manifolds) of different dimensions,
which is of  interest pertaining  to recent studies on heat kernels.



\end{abstract}

\section {Bauer diffusions and harmonic   kernels}\label{intro}

Darning of holes for symmetric Hunt processes by Dirichlet forms has been introduced in \cite{chen-fuku-book} 
and used in \cite{chen-fuku, chen-fuku-rohde}. 
In this paper, we shall  study darning  of   compacts (darning and gluing of finite unions of compacts),
which are not thin at any of their points, in a~potential-theoretic framework which may be described,
analytically,  in terms of harmonic kernels/harmonic functions or,  probabilistically, in terms of a~diffusion.\footnote{The 
author had the pleasure to listen to a talk by Zhen-Qing Chen on ``Browian motion
on spaces with varying dimensions'' (University of Bielefeld, February 2014) which motivated him
to try an approach  not using Dirichlet forms. 
 Having written most of this manuscript the author learnt that darning is also discussed in \cite{chen}
and will be the subject of \cite{chen-lou}.} 
We intend to accomplish this without leaving our kind of setting 
(for simplicity, we assume that we have a~base of regular sets and the constant function~$1$ is harmonic)
so that the procedure can be iterated without any problem. Our construction is of particular interest
with regard to recent studies of heat kernels. 

\begin{example}\label{simple-ex}
{
A basic example to start with is a~direct {\rm(}topological\,{\rm)} sum~$X$ of open sets $X_1, \dots,X_m$, $m\in\nat$, where each $X_j$ is
a~subset of some $\real^{d_j}$, $d_j\in\nat$, equipped with the classical harmonic structure {\rm(}Brownian motion{\rm)}
of dimension~$d_j$, and $K=K_1\cup\dots\cup K_m$, each~$K_j$ being a~compact {\rm(}in the simplest case, a~ball{\rm \,)} in $X_j$.
The aim is to identify $K$ with a~single point  $x_0$ and to define a~harmonic structure {\rm(}a diffusion{\rm)} 
on the new space such that, restricting on sets away from $K$,  we get what we had there before.
}
\end{example}

In the following, let $X$ be  a~locally compact (Hausdorff) space with countable base. 
Let us state right away that $X$ should also be locally connected; in fact, this is a~consequence 
of~$(H_4')$ below (see \cite[Satz 1.1.10]{Bauer66}).
Let $\U$ (resp.~$\U_c)$ denote the family of all open 
(resp.~relatively compact open) sets in $X$. For 
$\V\subset \U$ and $U\in \U$, let $\V(U)$ denote the set of all
$W\in \V$ such that $\ov W\subset U$.

Further, 
let $\B$ denote the set of all Borel sets in $X$
and let $\M$ denote the set of all measures on~$(X,\B)$.
 For  every $U\in \U$, let
$\B(U)$ (resp.~$\C(U)$) be the set of all $\B$-measurable (resp.\ continuous real)
functions on~$U$, and let $\C_0(U)$ be the set of all functions in $\C(U)$ which tend to zero
at $\partial U$.  As usual, given a~set $\F$ of functions, let~$\F^+$  (resp.~$\F_b$) denote
the set of all functions in $\F$ which are positive (resp.~bounded).   
Sometimes we shall tacitly identify functions 
on subsets $A$ of $X$ with functions on~$X$ taking the value $0$ on~$X\setminus A$.

A \emph{kernel on~$X$} is a~mapping $T\colon X\times \B\to [0,\infty]$ such that
$T(x,\cdot)\in \M$, $x\in X$,  and the functions $x\mapsto T(x,B)$, $B\in \B$, are $\B$-measurable. 
For every kernel~$T$ on~$X$, we define
\begin{equation*} 
                            Tf(x):=\int f(y) \,T(x,dy), 
\end{equation*} 
whenever $f\in \B(X)$, $x\in X$,  and the integral makes sense. 
For kernels $S,T$ on~$X$,  the kernel $ST$ is defined by $(S T)(x,\cdot):=\int T(y,\cdot)\,S(x,dy)$, that is,
\begin{equation*}
  (S T) f =S (T f), \qquad f\in \B^+(X).
\end{equation*} 
A kernel $T$ on~$X$ is  called a~\emph{boundary kernel
for~$U\in \U$} provided $T(x,X\setminus \partial U)=0$ for every  $x\in U$ and $T(x,\cdot)=\vx$
(Dirac measure at $x$) for every $x\in U^c$. 

\begin{example}\label{ex-diff}
If $\mathfrak X=(\Omega, X_t,\theta_t, \mathbbm P^x)$ is a~diffusion on~$X$ {\rm(}that is, a~Hunt process on~$X$
with continuous trajectories $t\mapsto X_t(\omega)${\rm)}, the \emph{exit kernels} ~$H_U$,  
$U\in \U$, given~by
\begin{equation}\label{diff-kernel}
           H_U(x,B):=\mathbbm P^x[X_{\tau_U}\in B, \tau_U<\infty], \qquad x\in X, B\in \B,
\end{equation} 
where $\tau_U:=\inf\{t\ge 0\colon X_t\in U^c\}$ denotes the \emph{exit time of~$U$},
are boundary kernels. 

These kernels are compatible, that is,
\begin{equation}\label{comp}
                 H_VH_U=H_U \qquad\mbox{ if  $V,U\in \U$  with $V\subset U$},
\end{equation} 
because of the trivial identity $\tau_U=\tau_V+ \tau_U\circ\theta_{\tau_V}$ and the strong Markov property.
If $U,V\in \U$ are disjoint, then obviously $H_V(x,\cdot)=H_{U\cup V}(x,\cdot)$ for every $x\in V$.
Let
\begin{equation*}
                               p_U(x):=\mathbbm E^x\tau_U, \qquad x\in U\in \U.
\end{equation*} 
If $p_U(x)<\infty$, then, of course, $H_U(x,X)=\mathbbm P^x[\tau_U<\infty]=1$.
\end{example}

\begin{definition}
A diffusion $\mathfrak X$ on~$X$ is called a~\emph{Bauer diffusion}, if there exists a~base $\V$ 
for the topology of $X$ such that for all $V\in \V $ the following holds. 
\begin{itemize}
\item[\rm $(D_1)$] The set $V$ is relatively compact and $p_V\in \C_0(V)$.
\item[\rm $(D_2)$]  For all $z\in \partial V$ and neighborhoods $W$ of $z$, $\lim_{x\to z} H_V(x,W^c)=0$,
and  $(H_Vf)|_{V}\in \C(V)$ for all $f\in \B_b(X)$. 
\end{itemize} 
\end{definition}

For an analytic discussion we shall consider a~Bauer family of harmonic boundary kernels
(cf.\ \cite[VIII.1]{BH}).
For the moment,  let us assume we have a~base  $\V\subset \U_c$  for the topology of $X$ and, for every $V\in \V$,
 a~boundary kernel $H_V$. 
Let $U\in \U$ and  let~$\W$ be  a~\emph{local choice from~$\V$ in~$U$},
that is, a~mapping $\W\colon x\mapsto \W(x)$, $x\in U$, where each~$\W(x)$ is a~fundamental system of neighborhoods $W\in \V(U)$ of $x$.
  Then a~lower semicontinuous
function $u\colon U\to \left]-\infty,\infty\right] $ is  called \emph{$\W$-hyperharmonic on~$U$}~if 
\begin{equation*} 
                                H_Wu(x)\le u(x) \qquad \mbox{ for all }x\in U\mbox{ and }W\in \W(x), 
\end{equation*} 
and the set of all $\W$-hyperharmonic functions on~$U$ is denoted by  $\H_{\W}^\ast(U)$. 

If $\W(x)$ is the set of \emph{all} $W\in \V(V)$ such that $x\in W$,
we drop the subscript $\W$ and have the set $\H^\ast(U)$  of \emph{hyperharmonic functions on~$U$}.
The set of all \emph{harmonic functions on~$U$} is defined by
\begin{equation}\label{def-harmonic}
        \H(U):=\{h\in \C(U)\colon H_Vh=h 
\mbox{ for }
V\in \V(U)\}=\H^\ast(U)\cap (-\H^\ast(U)). 
\end{equation}

\begin{definition}\label{hake}
Let $\V\subset \U_c$ be a~base for the topology of $X$ and let $H_V$, $V\in \V$,
be boundary kernels for $V$. Then
the family  $(H_V)_{V\in \V}$ is a~\emph{normal\footnote{The attribute``normal'' refers to $H_V1=1$.} 
Bauer family of  harmonic boundary kernels on~$X$}
 provided the following holds for every $V\in \V$.
\begin{itemize}
\item[\rm $(H_1)$] 
If $W\in \V$ and $W\subset V$, then
$  H_W H_V=H_V$.
\item [\rm $(H_2)$] 
For every $f\in \C(X)$, $H_Vf\in \C(X)$,  and $H_V1=1$. 
\item[\rm $(H_3)$] There exists a~local choice $\W$ from~$\V$ on~$V$
 such that ${}^+\H_{\W}^\ast(V)$ separates the points of $V$.
\item[\rm  $(H_4)$] For all   $f\in \B_b(X)$, the function $H_V f$ is continuous on~$V$. 
\end{itemize} 
\end{definition}

Having Corollary \ref{nearly} below it is easily established  that $(H_4)$ amounts to the \emph{Bauer convergence property}: 
\begin{itemize}
\item[\rm$(H_4')$] For every $U\in \U$ and every  bounded sequence $(h_n)$ in $\H(U)$ which is increasing, $\sup h_n\in \H(U)$.
\end{itemize}

For a~converse of the next result see Theorem \ref{kernel-to-diff}.

\begin{proposition}\label{diff-hake}
Let $\mathfrak X$ be  Bauer diffusion on~$X$ and $\V\subset \U_c$ be a~base for the topology of $X$
such that the exit kernels $H_V$, $V\in\V$,  satisfy $(D_1)$ and $(D_2)$. Then,
for all $U,V\in \V$ with $\ov U\subset V$,
\begin{equation}\label{huv}
 p_V, H_U p_V\in \H^\ast(V) \und p_V-H_Up_V=p_U>0 \mbox{ on }U.
\end{equation} 
In particular,  $(H_V)_{V\in \V}$ is a~normal Bauer family of harmonic boundary kernels.
\end{proposition}

\begin{proof} Clearly,  $(H_1)$, $(H_2)$ and $(H_4)$ hold. 
Let us fix  $U,V\in \V$ such that $\ov U\subset V$.
If $x\in W\in \V(V)$, then, by the strong Markov property, 
\begin{equation*} 
H_Wp_V(x)=\mathbbm E^x[\mathbbm E^{X_{\tau_W}}\tau_V]
=\mathbbm E^x(\tau_V\circ \theta_{\tau_W})=\mathbbm E^x(\tau_V-\tau_W)= p_V(x)-p_W(x).
\end{equation*} 
Therefore $p_V\in  \H^\ast(V)$ and (taking $W=U$), 
\begin{equation} \label{HUp}
H_Up_V(y) =\mathbbm E^y (\tau_V-\tau_U)=p_V(y)-p_U(y)\qquad\mbox{ for every } y\in V
\end{equation} 
(the equality holds trivially for $y\in V\setminus U$). So, by the strong Markov property, 
\begin{equation*} 
H_W H_Up_V(x)=\mathbbm E^x[\mathbbm E^{X_{\tau_W}}(\tau_V-\tau_U)] =\mathbbm E^x(\tau_V-\tau_W) -\mathbbm E^x(\tau_U\circ \theta_{\tau_W})
\end{equation*} 
for all $x\in W\in \V(V)$. Since $\tau_U\le \tau_W+\tau_U\circ \theta_W$, we see that $H_WH_Up_V(x)\le H_Up_V(x)$. 
Knowing that $H_Up_V\in \C_0(V)$, by $(H_2)$, we conclude that    $H_Up_V\in \H^\ast(V)$.

Finally, let $x,y\in V$, $x\ne y$. We choose $U\in \V(V)$  such that  $x\notin U$, $y\in U$.  Then
 $H_Up_V(x) =p_V(x)$ and $H_Up_V(y)<p_V(y)$, by (\ref{HUp}). So $p_V(x)\ne p_V(y)$ or~$H_Up_V(x)\ne H_Up_V(y)$. Thus $(H_3)$ holds
and the proof is finished. 
\end{proof}

In the remaining part of this section we shall present some facts on harmonic kernels and harmonic
functions which, in essence, are taken from \cite{Bauer66,BH,GH1,H-pert,H-strict,H-coupling} 
(the reader, who  to some extent is familiar with general potential theory, may directly pass to the next section and return to   
  this section as needed).

The weak separation property $(H_3)$ is ideal for capturing examples (if $U$ is an open set in $\real^n$, $\W(x)$ may
consist of balls centered at $x$) and it is strong enough for a~minimum principle (Proposition \ref{min-principle})
which, having $(H_1)$ and $(H_2)$,  implies that, in fact,  $\H_{\W}^\ast(U)=\H ^\ast(U)$
even for all $U\in \U$ and all local choices $\W$ from $\V$ on $U$ (Corollary~\ref{nearly}).

For the moment, let us only assume that $(H_V)_{V\in \V}$ satisfies $(H_1)$ and $(H_2)$. 
Of course, $(H_Vf)|_V\in \H(V)$ for every $f\in \C(X)$ (even for every $f\in \B_b(X)$ provided $(H_4)$ holds).

\begin{proposition} \label{min-principle}
Let $U\in\U_c$ and let $\W$ be a~ local choice  from $\V$ on $U$.
Suppose that ${}^+\H_{\W}^\ast(U)$ separates the points of $U$,
and let $v\in \H_{\W}^\ast(U)$ such that $\liminf_{x\to z} v(x)\ge 0$
for every $z\in \partial U$. Then $v\ge 0$ on $U$.
\end{proposition} 

\begin{proof} Immediate consequence of Bauer's  general minimum principle
(see \cite{bauer-silov} and \cite[p.\,7 and Korollar 1.3.7]{ Bauer66}).
\end{proof} 

\begin{corollary}
\label{nearly}
Let $U\in \U$ and suppose that there exists a~local choice $\W$ from~$\V$ on~$U$
such that ${}^+\H_{\W}^\ast(U)$   separates the points of~$U$. Then, for \emph{every} 
local choice~$\W$ from~$\V$ on~$U$,
\begin{equation*}
\H_{\W}^\ast(U)=\H^\ast(U).
\end{equation*}
In particular, ${}^+\H^\ast(U)$ separates the points of $U$.
\end{corollary} 

\begin{proof}[Proof {\rm(cf.\ the proof of {\cite[Satz 1.3.8]{Bauer66}})}]Let  $\W$ be a~local choice from $\V$ on $U$. 

(a) We  suppose first that   ${}^+\H_{\W}^\ast(U)$ separates the points of $U$. Let $u\in  \H_{\W}^\ast (U)$
 and  $V\in \V(U) $.
There are functions $f_n\in \C(\ov V)$, $n\in\nat$,  with $f_n\uparrow u$
on $\ov V$ as $n\to\infty$. For the moment, let us fix $n\in\nat$. By $(H_1)$, 
\begin{equation*} 
v_n:=(u-H_Vf_n)
\end{equation*}
 satisfies
$H_Wv_n(x)\le v_n(x)$ for all $x\in V$ and $W\in \W(x)$ with $\ov W\subset V$. 
Moreover,  $v_n$~is lower semicontinuous on $\ov V$ and $v_n=u-f_n\ge 0$ on $\partial V$. 
Hence $v_n\ge 0$ on~$V$, by Proposition~\ref{min-principle}.   Letting $n\to\infty$ we obtain that $u-H_Vu\ge 0$.
So $u\in \H^\ast(U)$. In particular, ${}^+\H^\ast(U)$ separates the points of $U$.

(b) Now let $\W$ be an arbitrary  local choice from $\V$ on $U$.
Trivially $\H^\ast(U)$  is contained in $\H_{\W}^\ast(U)$, 
and  we finally conclude, by (a),  that $\H_{\W}^\ast(U)=\H^\ast(U)$.
\end{proof}

From now on let us suppose that $(H_1)$ -- $(H_4)$ are satisfied.
By Corollary~\ref{nearly},(ii), we immediately obtain 
that $\H:=\{\H(U)\colon U\in \U\}$ is a~\emph{sheaf of continuous functions on~$X$}, that is, for every $U\in \U$,
the set $\H(U)$ is a~linear subspace of $\C(U)$ and the following hold.
\begin{itemize}
\item[\rm (i)] If $U'\in \U$, $U'\subset U$ and $h\in \H(U)$, then $h|_{U'}\in  \H(U')$.
\item[\rm (ii)] If $U=\bigcup_{i\in I} U_i$, $U_i\in \U$,  and $h\in \B(U)$ such that $h|_{U_i}\in \H(U_i)$ for every 
$i\in I$, then $h\in \H(U)$.
\end{itemize} 
Moreover,  for all $U\in \V$ and $f\in \C(X)$,
the function $H_Uf$ is harmonic on~$U$, by $(H_1)$ and $(H_2)$; by Proposition \ref{min-principle},
it is the only function which is harmonic on~$U$ and equal to $f$ on~$U^c$.

A set $U\in \U_c$ is called ($\H$-)\emph{regular}, if every function $f\in \C(\partial U)$ possesses
a~unique continuous extension $H_f^U\in\C(\ov U)$ which is harmonic on~$U$, and if $H_f^U\ge 0$ 
for~$f\ge 0$. Let $\U_r$ denote
the family of all regular sets. By the preceding observation, $\V\subset \U_r$ and, for all $U\in \V$ 
and $f\in \C(X)$, 
\begin{equation}\label{V-reg}
   H_{f|_{\partial U}}^U=H_Uf \on \ov U.
\end{equation} 
For every $U\in \U_r$, the mapping $f\mapsto H_f^U$ defines a~boundary kernel which,
in view of (\ref{V-reg}), may be denoted by $H_U$. By (i), it is trivial that $H_VH_Uf=H_Uf$
for all $U,V\in \U_r$ with $V\subset U$ and $f\in \C(X)$. Therefore
\begin{equation*} 
           H_VH_U=H_U,\qquad\mbox{ whenever } U,V\in \U_r \mbox{ and } V\subset U.
\end{equation*} 
If $U_1,U_2\in \U_c$  are disjoint and $U_1\cup U_2$ is regular, then obviously both $U_1$ and $U_2$
are regular and 
\begin{equation}\label{comp-reg}
           H_{U_j}(x,\cdot)=H_U(x,\cdot)\qquad \mbox{ for every }x\in U_j, \, j=1,2.
\end{equation} 

A set $U\in \U$ is called a~\emph{$\mathcal P$-set} if ${}^+\H^\ast(U)$ 
separates the points of $U$ (see Proposition~\ref{equiv-potential} for the justification of the prefix ``$\mathcal P$'').
By Corollary \ref{nearly}, every $V\in \V$ is a~$\mathcal P$-set and the sets of functions $\H^\ast(U)$ and $\H(U)$,
$U\in \U$, are not changed if we replace~$\V$ by the family of \emph{all} regular $\mathcal P$-sets.

Clearly, every open subset of a~$\mathcal P$-set is \hbox{a~$\mathcal P$-set}.
Moreover, every union of two disjoint $\mathcal P$-sets is a~$\mathcal P$-set.

For our darning we need the \emph{barrier criterion for regular sets}.  
If $V\in \U_c$ and~$\ov V$ is contained in  a~$\mathcal P$-set, then $V$ is regular if (and only if),
for every $z\in \partial V$, there exists an open  neighborhood $W_z$ of $z$ and a~strictly positive function $s\in \H^\ast(V\cap W_z)$
such that $\lim_{x\to z} s(x)=0$ (see \cite[Satz 4.3.3]{Bauer66}). In Example \ref{ex-diff}, 
this means that $T_{V^c}:=\inf\{t>0\colon X_t\in V^c\}=0$ $\mathbbm P^z$-a.s.\ (see \cite[VI.4.16]{BH}). 

In particular, for any~$\mathcal P$-set $W$,  
the intersection of any two regular sets with closure in~$W$ is regular.
So our pair $(X,\H)$  not only satisfies  (i), (ii) and the Bauer convergence property holds, 
but there also exists a~base $\tilde \V$  for $\U$ 
consisting of $\H$-regular sets which is stable under finite intersections, that is,
$(X,\H)$ is a~\emph{Bauer space} (see \cite[\S 3.1]{Const} and \cite[7.1]{GH1}).

It can be shown that, conversely, for any   Bauer space $(X,\H)$ with $1\in \H(X)$,
there exists a~base $\V$ of regular sets such that the corresponding kernels $H_V$,
$V\in \V$, form a~normal Bauer family of harmonic kernels such that, for every $U\in \U$,
$\H(U)$ is the set of all corresponding harmonic functions on~$U$
(cf.~\cite[Section VIII.1]{BH}). 

We refer the reader to \cite[Section VIII]{BH} and \cite[Section 7.1]{GH1} for  examples given by
\begin{equation*} 
\H(U):=\{h\in \C(U)\colon Lh=0\}, \qquad U\mbox{ open set in }X ,
\end{equation*} 
where $L$ is a~partial differential operator of second order on an open set $X$  in $\real^d$.

Let $U\in \U$. A~function $s\in \H^\ast(U)$  is \emph{superharmonic on~$U$} if, for every $V\in \V(U)$,  
 the function~$H_Vs$ is  locally bounded on~$V$
(and hence, by $(H_4)$, is harmonic on~$V$).  Obviously, every locally bounded hyperharmonic function on~$U$
is superharmonic. Let $\mathcal S(U)$ denote the set of all superharmonic functions on~$U$.

A function  $s\in \mathcal S^+(U)$  is called \emph{potential on~$U$} if $0$
is the largest harmonic minorant of $s$ on~$U$. Let  $\mathcal P(U)$ denote the set of all 
potentials on~$U$. Every function $s\in \Su^+(U)$ admits a unique decomposition $s=h+p$ such that $h\in \H^+(U)$
and $p\in \mathcal P(U)$ (Riesz decomposition). A~potential $p$ on~$U$ is called \emph{strict}, if any two measures 
$\mu,\nu$ on~$U$ coincide provided $\int p\,d\mu=\int p\,d\nu<\infty$ and 
$\int v\,d\mu\le \int v\,d\nu$ for every $v\in {}^+\H^\ast(U)$. Of course, every strict potential
is strictly positive.
The following well known equivalences justify our notion of a~$\mathcal P$-set. 
\begin{proposition}\label{equiv-potential}
For every $U\in \U$, the following properties are equivalent.
\begin{itemize}
\item
 $U$ is a $\mathcal P$-set {\rm(}that is, ${}^+\H^\ast(U)$ separates the points of $U${\rm)}.
\item
There exists a  potential $p>0$ on~$U$.
\item
There exists a strict continuous real potential on~$U$.
\end{itemize} 
\end{proposition}

If $U,V\in \U$ and $p_U\in \mathcal P(U)$, $p_V\in \mathcal P(V)$, then $p_U$ and $p_V$
are called \emph{compatible}, if the difference $p_U-p_V$ is harmonic on~$U\cap V$.

Finally, let us suppose once more that, in  Example \ref{ex-diff},  we have a Bauer diffusion~$\mathfrak X$
and a base  $\V$ such that the exit kernels $H_V$, $V\in \V$, satisfy $(D_1)$ and $(D_2)$.
We know, by Proposition \ref{diff-hake}, that  $(H_V)_{V\in \V}$ is a normal  Bauer family of harmonic boundary kernels.
Moreover, we now see that the functions $p_V\in \H^\ast(V)\cap C_0(V)$, $V\in \V$, are  potentials
which are compatible, since, by (\ref{huv}),  $p_U-p_V=H_W(p_U-p_V)$ whenever $U,V,W\in \V$ such that $\ov W\subset U\cap V$.
It can be shown that every $p_V$ is a~strict potential   (cf.\ \cite[VI.7.8]{BH}). 

By \cite{H-pert}  in combination with \cite{H-strict}, which provides a local characterization of the fine support
of continuous real potentials, the  following converse holds. 

\begin{theorem}\label{kernel-to-diff}
Let $(H_V)_{V\in \V}$ be a normal Bauer family of harmonic kernels on~$X$ and let
$(p_V)_{V\in \V}$ be a compatible family of strict continuous real potentials
{\rm(}there always is such a~family{\rm)}. 

Then there exists a unique  {\rm(}Bauer{\rm)} diffusion $\mathfrak X$ on~$X$ 
such that, for every $V\in \V$, the kernel~$H_V$ is the exit kernel of $V$
and $\mathbbm E^x\tau_V=p_V(x)$, $x\in V$.
\end{theorem}

 \section{Darning of strongly stable compacts}

Let $(H_V)_{V\in \V}$ be a normal Bauer family of harmonic boundary kernels on~$X$ (see Definition~\ref{hake}).
 A compact $K$ in $X$ which is contained in a $\mathcal P$-set  is \emph{stable} if,   for every $z\in \partial K$,
 there exists an open neighborhood $W$
  of $z$ and a strictly positive function $s\in \mathcal S^+(W\setminus K)$ such that $\lim_{x\to z} s(x)=0$ ($T_K=0$
 $\mathbbm P^z$-a.s.\ for an associated diffusion).

\begin{definition} 
We say that a compact $K\ne \emptyset$ is \emph{strongly stable}, if it is contained in a $\mathcal P$-set $W$
and there exists a strictly positive function $h\in \H(W\setminus K)$ which tends to zero at $\partial K$. 
\end{definition}

Clearly,  every finite union of pairwise disjoint strongly stable compacts is strongly stable.
In classical potential theory of $\real^d$, a   compact $K\ne \emptyset$ is strongly stable if and only if it is stable
and has at most finitely many holes (that is, 
bounded connected components of~$\real^d\setminus K$).
For a more general statement see Proposition~\ref{components}.

Let us fix a strongly stable compact  $ K$ in $X$. 
We define a~new space $ X_0$
identifying~$K$ with a~point  $x_0\in K$, 
that is,
\begin{equation*} 
    X_0:=\{x_0\} \cup(X\setminus K) 
\end{equation*} 
such that the topology of $ X_0\setminus \{x_0\}$ is the topology of $X\setminus K$ 
and the neighborhoods of $x_0$ have the form 
\begin{equation*} 
                \{x_0\}\cup (N\setminus K),
\end{equation*} 
where   $N$ is a~neighborhood of $K$ in $X$. 

Our aim is the construction of a normal Bauer family of harmonic boundary kernels on~$X_0$ containing
the kernels $H_V$, $V\in \V(X\setminus K)$.
We shall then say that it is obtained by darning of $K$ (with $\{x_0\}$). 
Doing this for  a~union $K$ of pairwise disjoint stable compacts $K_1,\dots,K_m$
we shall also speak of \emph{darning and gluing of~$K_1,\dots,K_m$}.

To that end  we choose once and for all  a~regular $\mathcal P$-set  $W_0$ with  $K\subset W_0$ and 
\begin{equation}\label{def-g}
                      g:=H_{W_0\setminus K} 1_{\partial W_0}>0 \on W_0\setminus K.
\end{equation} 
Such a  set exists. Indeed, let $W$ be a $\mathcal P$-set containing $K$ and $h\in \H(W)$, $h>0$,
such that $h$ tends to zero at $\partial K$.
Let us choose an  open neighborhood $W'$ of $K$ which is relatively compact in $W$, 
and define $a:=\inf \{h(x)\colon x\in \partial W'\}$,
\begin{equation*}
                                  W_0:=K\cup \{x\in W'\setminus K\colon h(x)<a/2\}. 
\end{equation*} 
Then both $W_0$ and $W_0\setminus K$ are regular by the barrier criterion (for the points
in $\partial W_0$ consider the function $(a/2)-h$). Obviously, $H_{W_0\setminus K}1_{\partial W_0}=(a/2)\inv h$
on $W_0$. 

For  $0<r <1$, let 
\begin{equation*} 
               A_{r}:=\{x\in W_0\setminus K\colon g(x)<r\} \und                S_{r}:=\{x\in   W_0\setminus K\colon g(x)=r\}.
\end{equation*} 
The sets $A_{r}$ are regular, since, by the barrier criterion, also the points in $S_{r}$ are regular points of $A_{r}$, $0<r<1$.
The  family of sets   
$$
            U_r:=\{x_0\} \cup A_{r}, 
\qquad 0<r<1,
$$
is a~fundamental system of   open neighborhoods of $x_0$ such that $\partial U_r= S_{r}$.
Given  $0< t<1$ and a~probability measure $\sigma$ on~$A_t$, we define  
\begin{equation*} 
 \sigma H_{A_t}:= \int H_{A_t}(y,\cdot)\,d\sigma(y),
\end{equation*} 
and note that $ \sigma H_{A_t}(S_t)=r/t$ if $\sigma$ is supported by $S_r$ (see (\ref{sat})). 

Now let us specify the first of our main results (for its proof see Section \ref{main-proof}).

\begin{theorem}\label{main}
Let $\sigma_r$ be probability measures on~ $S_r$, $0<r<1$, and         
\begin{equation*}
          \V_0:=\{U_r\colon 0<r<1\}\cup \{V\in \V\colon \ov V\cap K=\emptyset\}. 
\end{equation*} 
\begin{itemize} 
\item[\rm (1)] 
The following two statements are equivalent.
\begin{itemize} 
\item[\rm (a)] 
There exist  boundary kernels $H_{U_r}$, $0<r<1$,  such that $(H_V)_{V\in \V_0}$ is a~normal Bauer 
family of harmonic boundary kernels on~$X$ {\rm(}with an associated sheaf~$\H_0$ of harmonic functions{\rm)} and 
\begin{equation}\label{mx0} 
                                 H_{U_t}(x_0,\cdot) =\sigma_t  \qquad \mbox{ for every }0<t<1.
\end{equation} 
\item[\rm (b)]
The probability measures $\sigma_r$ are \emph{compatible}, that is, for all $0<r<t<1$,
\begin{equation}\label{compatible}
\sigma_{r}H_{A_t }=  ( r/t)\sigma_{t}  \on S_t.
\end{equation} 
\end{itemize} 
\item[\rm (2)] Suppose that $(a)$ in $(1)$ holds. Then, for all $0<t<1$ and $B\in \B$,
\begin{equation}\label{HUr}
H_{U_t}(y,B)=H_{A_t}(y, B\cap S_t)+ (1- (g(y)/t))\sigma_t(B), \quad y\in A_t.
\end{equation} 
The singleton $\{x_0\}$ is strongly stable   and  $\{x_0\}\cup (W_0\setminus K)$ is a $\mathcal P$-set for $(X,\H_0)$.

Further,  for all open neighborhoods $W$ of $x_0$ and 
 $h\in \C(W)$ which are \hbox{{\rm ($\H$-)}} harmonic on~$W\setminus \{x_0\}$, the following properties are equivalent:
\begin{itemize}
\item[\rm (a)] The function $h$ is {\rm ($\H_0$-)}harmonic on~$W$.
\item[\rm (b)] For every $0<r<1$ with $\ov U_r\subset W$,  $\int h\,d\sigma_r=h(x_0)$. 
\item[\rm (c)] There exists $0<r<1$ with $\ov U_r\subset W$ and  $\int h\,d\sigma_r=h(x_0)$. 
\end{itemize} 
\end{itemize} 
\end{theorem} 

\begin{remark} \label{simple-sigma}
{\rm
 For the moment, let us consider Example \ref{simple-ex} and assume that 
 the compacts  $K_1,\dots, K_m$ are closed balls having centers $x_1,\dots,x_m$, respectively. 
Then it is easy to get compatible probability measures $\sigma_r$ provided $W_0$ has been chosen 
to consist of open balls $B_j$ in $X_j$ which are centered at $x_j$, $1\le j\le m$. Indeed, then each $S_r$, $0<r<1$,
is a~union of spheres $S_{j,r}$ contained in the shells $B_j\setminus K_j$, $1\le j\le m$. 
Denoting the normed surface measures on $S_{j,r}$ by $\sigma_{j,r}$, 
  it suffices to 
fix $\a_1,\dots,\a_m\in [0,1]$ with  $\a_1+\dots+\a_m=1$ and to take
\begin{equation*}
                              \sigma_r:=\a_1 \sigma_{1,r}+ \dots \a_m \sigma_{m,r},\qquad 0<r<1.
\end{equation*} 
}
\end{remark}

In the general case, the existence of compatible probability measures $\sigma_r$ on~$S_r$, $0<r<1$, 
follows from $(H_1)$ and the weak compactness of probability measures on a compact set.

\begin{proposition}\label{ex-comp}
Let   $(\eta_n)$ be a sequence in $(0,1)$ which  is strictly decreasing to zero.
For every $n\in\nat$, let $\nu_n$ be a probability measure on~$S_{\eta_n}$. 
Then there exists a~subsequence of $(\nu_{n_k})$ of $(\nu_n)$ such that, for every $0<r<1$,
the probability measures~$\sigma_{r,k}$ on~$S_r$, defined by
\begin{equation*}
                         \sigma_{r,k}:= (r/\eta_{n_k}) (\nu_{n_k} H_{A_r})|_{S_r}, \qquad\mbox{ if }  \eta_{n_k}<r, 
\end{equation*} 
converge weakly to a   probability measure $\sigma_r$ as $k\to\infty$. 
The measures $\sigma_r$, $0<r<1$, are compatible.
\end{proposition}

So, by Theorem \ref{main}, a darning of $K$ is always possible (we could take a Dirac measure on each $S_{\eta_n}$).
However, looking at examples, where $K$ is a disjoint union of compacts or contains holes,
it is clear that some controlled choice of the measures~$\nu_n$
is necessary to obtain what we really want  (see  the examples at the beginning of Section \ref{appendix}). 
Otherwise, it may happen that, for every $0<r<t<1$,
the exit measure $H_{U_r}(x_0,\cdot)=\sigma_r$ charges only one of possibly many connected components
of $A_t$. However,  in an elliptic situation,
we would ask for a~darning such that the outcome is elliptic as well.
And we may even desire some numerical control of the distribution of the 
measures $\sigma_r=H_{U_r}(x_0,\cdot)$ on various parts of~$S_r$. 

To achieve this, let us first observe that, for every $0<t<1$, 
every connected component $V$ of $A_t$ has boundary points both in $K$
and in $S_t$. This implies the following.

\begin{lemma}\label{finite}
Each set $A_t$, $0<t<1$, has only  finitely many connected components. 
\end{lemma}

\begin{proof} Assume on the contrary that we have $0<t<1$ such that $A_t$
has countably many connected components $V_1,V_2,V_3,\dots$. We may choose
points $x_n\in V_n$ such that $g(x_n)=t/2$. Let $(x_{n_k})$ be convergent subsequence,
$x:=\lim_{k\to\infty} x_{n_k}$. Then $g(x)=t/2$, and hence $x\in V_n$ for some $n\in\nat$.
So there exists $k_0\in\nat$ such that, for every $k>k_0$,  $x_{n_k}\in V_n$,
and hence $V_{n_k}=V_n$. A contradiction.
\end{proof} 

Let $\cc_0$ be the family of connected components of $W_0\setminus K$
and, for~$n\in\nat$, let~$\cc_n$ denote the family of connected components of $A_{\eta_n}$.
Obviously,\begin{equation*}
                                \bigcup_{V'\in \cc_{n}, V'\subset V} V' =V\cap A_{\eta_n} \qquad
\mbox{ for all $n\in\nat$ and $V\in \cc_{n-1}$}.
\end{equation*} 
We successively choose $\a_V\in [0,1]$ (if we wish, strictly positive) such that
\begin{equation}\label{choice-a}
     \sum_{V\in \cc_0} \a_V=1 \und     \sum_{V'\in \cc_n, V'\subset V} \a_{V'} =\a_V \ \mbox{  for } n\in\nat, \, V\in \cc_{n-1}.
\end{equation} 
Finally, we fix   probability measures $\nu_n$ on~$S_{\eta_n}$, $n\in\nat$, such that 
\begin{equation}\label{choice-nun}
          \nu_n(V\cap S_{\eta_n})=\a_V\qquad \mbox{ for every }V\in \cc_{n-1}
\end{equation} 
(which is possible, since the sets $V\cap S_{\eta_n}$ are not empty).

Before writing down the consequences of these choices, let us recall
that $(H_V)_{V\in \V}$ is  \emph{elliptic} if one of the following 
equivalent statements holds.
\begin{itemize} 
\item Every harmonic function $h\ge 0$ on a connected open set  $U$
 is identically zero or   strictly positive on~$U$.
\item Every hyperharmonic function $h\ge 0$ on a connected open set  $U$
 is identically zero or   strictly positive on~$U$.
\item For all connected regular sets~$U$
  and  $x\in U$,   the support of $H_U(x,\cdot)$ is $\partial U$.
\end{itemize}

\begin{theorem}\label{fine-choice}
The preceding choices yield compatible probability measures $\sigma_r$, $0<r<1$,
such that {\rm(}taking $\eta_0:=1${\rm)}
\begin{equation*}
                                 \sigma_r(V)=\a_V\qquad \mbox{ for all  $ V \in \cc_n$, $n\ge 0$, with $r<\eta_n$.}
\end{equation*}
In particular, $(H_V)_{V\in \V_0}$ is elliptic on~$X_0$, if $(H_V)_{V\in \V}$ is elliptic and all $\a_V>0$.
\end{theorem} 

\begin{proof} Let $r,n,V$ be as indicated and let $k\in\nat$ with $\eta_{n_k}<r$. 
 By induction, (\ref{choice-a}) and (\ref{choice-nun}) imply that $\nu_{n_k}(V)=\a_V$. 
By  (\ref{comp-reg}),
\begin{equation*} 
\sigma_{r,k}(V\cap S_r)=(r/\eta_k)(\nu_{n_k}|_V) H_{A_r\cap V}(S_r)=\a_V,
\end{equation*} 
since both $A_r\cap V $ and $A_r\setminus V$ are open sets. The proof is completed letting 
$k\to \infty$ (the sets $S_r\cap V$ and $S_r\setminus V$ are closed).
\end{proof} 
 
Finally, we see that starting with a Bauer diffusion $\mathfrak X$ on~$X$ and an arbitrary neighborhood 
$N$ of the strongly stable compact~$K$, we may construct a darning of~$K$ that yields  a Bauer
diffusion~$\mathfrak X_0$ on~$X_0$ which, outside of $N$, is equivalent to~$\mathfrak X$.

\begin{corollary}\label{diff-diff}
Suppose that $\mathfrak X$ is a Bauer diffusion on~$X$, $\V$ is a base for the topology of $X$,
 and $(H_V)_{V\in \V}$ is a family of exit kernels satisfying $(D_1)$ and $(D_2)$. As before, let $(H_V)_{V\in\V_0}$
be a Bauer family of harmonic boundary kernels on~$X_0$ obtained by   darning of the strongly
stable compact $K$.

Then, for every $0<r<1$,
 there exists a Bauer diffusion $\mathfrak X_0$ on~$X_0$ such that, for every $V\in \V_0$, $H_V$ is the exit kernel
for $V$ with respect to $\mathfrak X_0$ and  the restrictions of $\mathfrak X$ and $\mathfrak X_0$ on~$X\setminus (K\cup \ov A_r)$ 
have the same finite-dimensional  distributions.
\end{corollary}

\begin{proof} Let us recall that, for all $U\in \U$ and $p\in \mathcal P(U)\cap C(U)$,
there exists a unique kernel $K_p$ on~$U$ such that $K_p1=p$ and, for every $f\in \B_b^+(U)$,
$f\circ p:=K_pf$ is a~potential on~$U$ which is harmonic outside the support of $f$.

In the terminology of \cite{H-pert}, $(p_V)_{V\in \V}$ corresponds to some
  $M\in\Gamma^+(X,\Q)$, a section in the sheaf of germs of   potentials on~$X$, which
is strict.
Let $\vp\in \C(X)$, $0\le \vp\le 1$ on~$X$, $\vp=0$ on~$K\cup A_{r/2}$,
$\vp=1$ on~$X\setminus (K\cup A_{3r/4})$, and let $M':=\vp\circ M$, that is, $M'$ is the section corresponding
to $(\vp\circ p_V)_{V\in \V}$. Then $M'=0$ on~$K\cup A_{r/2}$, and we may consider $M'$ as element of $\Gamma^+(X_0,\Q_0)$
which vanishes on $U_{r/2}$.

Further, let $q\in \C(U_r)$ be a strict potential on~$U_r$ (with respect to $(H_V)_{V\in \V_0}$),
and let $M''\in \Gamma^+(X_0,\Q_0)$ be the section corresponding to $\vp_0\circ q$, where $\vp_0(x_0):=1$
and $\vp_0:=1-\vp$ on~$X\setminus K$. Then $M''=0$ on~$X_0\setminus U_r$ and $M_0:=M'+M''\in \Gamma^+(X_0,\Q_0)$
is strict, by~\cite{H-strict}.

 By Theorem \ref{kernel-to-diff}, there exists a~Bauer diffusion $\mathfrak X_0$
on~$X_0$ corresponding to $M_0$. An application of results in \cite{H-pert}, in particular Theorem 4.16,  completes the proof.
\end{proof} 

\begin{remark} {\rm
In  Example \ref{simple-ex} we may achieve
that the restrictions of~$\mathfrak X$ and~$\mathfrak X_0$ on $X\setminus (K_1\cup\dots\cup K_m)$ 
have the same finite-dimensional
distributions provided  each~$K_j$ is a closed ball 
in $X_j$ with center $x_j$  (we already observed in Remark \ref{simple-sigma}
that then a straightforward choice of compatible probability measures~$\sigma_r$, $0<r<1$,  is possible; but this will be 
of no importance here). 

Let $1\le j\le m$, let~$\V_j$ be the set  of all open balls $B$ with $\ov B\subset X_j$ and
let $p_V$, $V\in \V_j$, be given by the expected exit times for Brownian motion on $\real^{d_j}$.
Further,  let $\W_j$ denote  the set of all $B\in \V_j$ with center  $x_j$ which contain  $K_j$. 
If $\g_0>0$ is sufficiently small,  then, for every $0<\g<\g_0$, there exist $W_j(\g)\in \W_j $  such that 
\begin{equation*}
                                 p_{W_j(\g)}   =\g   \on  \partial K_j .
\end{equation*} 
The sets 
\begin{equation*} 
                              W(\g):=\{x_0\}\cup \bigcup\nolimits_{1\le j\le m} (W_j(\g)\setminus K_j), \qquad 0<\g <\g_0,
\end{equation*} 
form a~fundamental system $\W_0$ of open neighborhoods of $x_0$ in $X_0$ and the functions 
\begin{equation*}
                                 p_{W(\g)}(x):=\begin{cases}
                                                  \g,&\quad x=x_0,\\
                                                  p_{W_j(\g)}(x),&\quad x\in W_j(\g)\setminus K_j. 
                                                 \end{cases}
\end{equation*} 
are continuous strict potentials on $W_\g$. Finally, let
\begin{equation*}
                                   \V:=\W_0\cup\bigcup\nolimits_{1\le j\le m}\{V\in \V_j\colon V_j\cap K_j=\emptyset\}.
\end{equation*} 
Then an application of Theorem \ref{kernel-to-diff} and  results in \cite{H-pert}  yields a~Bauer diffusion $\mathfrak X_0$
on~$X_0$ having the desired properties.
}
\end{remark}

\section{Proof of Theorem \ref{main}}\label{main-proof}

Let us first note that, for all $0 <t<1$ and $x\in A_t$,   
\begin{equation}\label{hAr} 
         H_{A_t}(x, S_t)  =g(x)/t,\qquad
          H_{A_t}(x, \partial K)=1-         H_{A_t}(x, S_t)  =1-(g(x)/t). 
\end{equation} 
In particular, for every $0<r<t$ and every  probability measure $\sigma$ on~$S_r$,  
\begin{equation}\label{sat}
\sigma H_{A_t} (S_{t}) :=\int H_{A_t}(x,S_t)\,d\sigma(x)=r/t.
\end{equation}

\begin{proof}[Proof of Theorem \ref{main}]   For the moment, let us fix  $h\in \C(X_0)$ and $0<t<1$. Then
$\lim_{x\in A_t, x\to z} h(x)=h(x_0)$ for every $z\in \partial K$. So, by (\ref{hAr}), $h|_{A_t}\in \H(A_t)$ 
if and only if 
  \begin{equation}\label{hy}
         h(x)= H_{A_t} (1_{S_t} h)(x)+     (1-t\inv g(x)) h(x_0) \qquad\mbox{ for every } x\in A_t.
\end{equation}         

(1)  Let us  suppose now that $(a)$ in  $(1)$ holds. Let  $0<t<1$,
 $f\in \C(X)$ and define $h:=H_{U_t} f$. Then,  by (\ref{mx0}),
\begin{equation*} 
h(x_0)=\int f\,d\sigma_t.
\end{equation*} 
 Moreover,  $h|_{A_t}\in \H_0(A_t)=\H(A_t)$, 
and hence (\ref{hy}) holds. Since $h=f$ on~$S_t$, this proves (\ref{HUr}).

Further, let us consider   $0<r< t<1$. By $(H_1)$, $H_{U_r}h=h$, and hence
 $ h(x_0)=\int h\,d\sigma_r$,  by (\ref{mx0}).
Integrating (\ref{hy}) with respect to $\sigma_r$ we obtain that
\begin{equation*}
 h(x_0)=\int h\,d\sigma_r=\int 1_{S_t}f d(\sigma_r H_{A_t}) + (1-(r/t)) h(x_0),
\end{equation*} 
and therefore
\begin{equation*} 
 \int (1_{S_t}f) d(\sigma_r H_{A_t})= (r/t) h(x_0)=(r/t) \int f\,d\sigma_t.
\end{equation*} 
 Thus (\ref{compatible}) holds. 

Next let us suppose conversely that (b) holds. For $0<t<1$, we (have to) define   boundary kernels $H_{U_t}$ on~$X_0$  by
(\ref{HUr}), that is, for all $f\in \B_b(X)$,   $H_{U_t} f=f$ on~$X_0\setminus U_t$ and
\begin{equation*} 
    H_{U_t}f(x)=\begin{cases} \int f\,d\sigma_t,&\quad x=x_0, \\
                               H_{A_t}(1_{S_t} f)(x)+   (1- (g(x)/t)) \int f\,d\sigma_t,&\quad x\in A_t.
                      \end{cases}
\end{equation*} 
Then $H_{U_t}f$ is harmonic on $A_t$. 
For every $z\in \partial K$,   
\begin{equation*} 
\lim_{x\to z}  H_{A_t}(1_{S_t} )(x)=\lim_{x\to z} g(x)/t = 0, 
\end{equation*} 
 and hence  
$H_{U_t}f$ is continuous at $x_0$. Moreover, $H_{A_t}(1_{S_t} f)$ is continuous on~$A_t$ and, if $f$ is continuous,
tends to~$f$ at~$S_t$. Therefore $H_{U_t}$ satisfies $(H_2)$ and $(H_4)$. 

If $0<r<t$ and $h:=H_{U_t}f$, then
\begin{equation*}
H_{U_r}h(x_0)=\int h\,d\sigma_r=\sigma_rH_{A_t}(1_{S_t} f)+(1-(r/t))\int f\,d\sigma_t=\int f\,d\sigma_t=h(x_0),
\end{equation*} 
by (\ref{compatible}). So $(H_1)$ holds. 

Of course, $\{x_0\}$ is strongly stable, since $g$ is a strictly positive harmonic function on~$W_0\setminus K$ and 
$\lim_{x\to x_0} g(x)=0$. To prove that $U_1:=\{x_0\}\cup (W_0\setminus K)$ is a $\mathcal P$-set,
and hence $(H_3)$ holds, we define  a~function $w_0$ on~$U_1$ by~$w_0(x_0)=1$ and $w_0=1-g$
on~$W_0\setminus K$ and consider
\begin{equation*}
\W(x):=\begin{cases}
\{U_r\colon 0<r<1\},&\quad x=x_0,\\
\{V\in \V(W_0\setminus K)\colon x\in V\},&\quad x\in W_0\setminus K.
\end{cases} 
\end{equation*} 
Clearly, $w_0\in \H_{\W}^\ast(U_1)$. Now let $x_1,x_2\in U_1$, $x_1\ne x_2$. 
 Of course, $w_0(x_1)\ne w_0(x_2)$ unless $x_1,x_2\in S_r$ for some $0<r<1$. 
So suppose that $0<r<1$ and $x_1,x_2\in S_r$. By assumption, $W_0$ is a $\mathcal P$-set for $(X,\H)$. Hence there exists
a~function $v\in {}^+\H^\ast(W_0)$ such that $v(x_1)\ne v(x_2)$. Without loss of generality 
$v\le r$ (take $a>v(x_1)\wedge  v(x_2)$ and replace $v$ by $r\wedge (r v/a)$). Let us define a function~$w_1$ on~$U_1$
by $w_1(x_0):=1$ and $w_1:= (w_0+v)\wedge1$ on~$W_0\setminus K$. Then $w_1\in \H_{\W}^\ast(U_1)$. Since $w_0=1-r$ on~$S_r$
and $v\le r$, we conclude that $w_1(x_1)=(1-r) +v(x_1) \ne (1-r) + v(x_2)=w(x_2)$.

Finally, to prove the equivalence of (a), (b) and (c) in (2),  let $W$ be an open set in $X_0$ and $0<t<1$ such that $\ov U_t\subset W$.
Moreover, let $h\in \C(W)$ such that $\int h\,d\sigma_t=h(x_0)$ and $h $ is harmonic on~$W\setminus \{x_0\}$. 
Then  $h(x_0)=H_{U_t}h(x_0)$, by (\ref{mx0}),  and both $h$ and $H_{U_t}h$ are harmonic on~$A_t$. Therefore $H_{U_t}h=h$ also on~$A_t$. 
In particular, $h|_{U_t}=(H_{U_t}h)|_{U_t}\in \H_0(U_t)$. 
So $h\in \H_0(W)$, by (ii). This shows that (c) implies (a). The implications $(a)\Rightarrow (b)\Rightarrow (c)$ are trivial.
\end{proof}

\begin{remark}{\rm
 It is a matter of taste if darning and gluing of pairwise disjoint strongly stable compacts $K_1,\dots,K_m$ is done in one step
or if, one at a time, each~$K_j$ is identified with a~(stable) point $x_j$, and then  gluing of the $m$~points $x_1,\dots,x_m$
(darning of $\{x_1,\dots,x_m\}$)  is performed.
}
\end{remark}

\section{Proof of Proposition \ref{ex-comp}}\label{section-compatible}

We begin with a~simple observation based on $(H_1)$.

\begin{lemma}\label{partial-compatible}
Let $0<\eta<1$ and let $\nu$
be any probability measure on~$S_{\eta}$. 
Then the   measures $\nu_r:=(r/\eta) (\nu H_{A_r})|_{S_r}$,  $\eta<r<1$, are compatible probability measures.
\end{lemma} 

\begin{proof} 
 Let $\eta<r<t<1$. By (\ref{sat}), $\nu_r$ is a probability measure. Since $\nu H_{A_r}$ is supported by $K\cup S_r$
and $K\subset A_t^c$, we know that $((\nu H_{A_r})|_K)H_{A_t}=(\nu H_{A_r})|_K$, and hence
$  (\eta/ r) \nu_r H_{A_t}|_{S_t} = (\nu H_{A_r}|_{S_r}) H_{A_t}|_{S_t} = (\nu H_{A_r} H_{A_t} )|_{S_t} =(\nu H_{A_t} )|_{S_t} 
 =( \eta/ t) \nu_t$.
 \end{proof}

\begin{proof}[Proof of Proposition \ref{ex-comp}]
For $m\in\nat$, every sequence of probability measures on~$S_{\eta_m}$ contains 
a~weakly convergent subsequence. The measures $(\eta_m/\eta_n) (\nu_n H_{A_{\eta_m}})|_{S_{\eta_m}}$, $n>m$, 
are probability measures,  by Lemma \ref{partial-compatible}.  So an obvious diagonal procedure
yields natural numbers $n_1<n_2<n_3<\dots$ such that, for every~$m\in\nat$,
the measures 
\begin{equation*}
          \nu_{m,k}:=(\eta_m/\eta_{n_k}) (\nu_{n_k} H_{A_{\eta_m}})|_{S_{\eta_m}}, \qquad k>m,
\end{equation*} 
weakly converge to a~probability measure $\mu_m$  on~$S_{\eta_m}$ as $k\to \infty$. 

Now let us fix $0<r<t<1$ and $m\in\nat$ such that $\eta_m<r$.  By Lemma \ref{partial-compatible},
\begin{equation*} 
                                \sigma_{r,k}=(r/\eta_m) (\nu_{m,k}H_{A_r})|_{S_r}  \und     \sigma_{t,k}=(t/r)( \sigma_{r,k}H_{A_t})|_{S_t}.
\end{equation*}  
Letting $k\to\infty$ we obtain $\sigma_r=(r/\eta_m) (\mu_mH_{A_r})|_{S_r}$ and $\sigma_t=(t/r)(\sigma_rH_{A_t})|_{S_t}$.
\end{proof}

\section{Appendix: Stable and  strongly stable sets}\label{appendix}

It may be instructive to look first at some examples in 
$X:=\real^d$, $d\ge 3$, for the classical case.
For $n\in\nat$,  let 
\begin{equation*}
 x_n:=(2^{-n}, 0,\dots,0), \quad  B_n:= \{x\in\real^d\colon |x|<2^{-(n+2)} \}, \quad A_n:=\ov B_n\setminus B_{n+1}.
\end{equation*} 
The set
\begin{equation*}
            K:=\{0\}\cup \bigcup\nolimits_{n\in\nat}  (x_n+\ov B_n) 
\end{equation*} 
is strongly stable. Taking $W_0:=\{x\in\real^d\colon |x|<1\} $ and $g:=H_{W_0\setminus K} 1_{\partial W_0}$
the number of connected components of $A_r=\{0<g<r\}$, $0<r<1$,
increases to $\infty$ as $r\to 0$. For every $m\in\nat$, 
\begin{equation*} 
                   L_m:=\{0\}\cup \bigcup\nolimits_{n\le m}  (x_n+A_n) \cup \bigcup\nolimits_{n> m} (x_n+\ov B_n) 
\end{equation*} 
is (still) strongly stable. However,
\begin{equation*} 
                   L:=\{0\}\cup \bigcup\nolimits_{n\in\nat} (x_n+A_n)  
\end{equation*} 
is stable,  but not strongly stable, since $\real^d\setminus L$ has infinitely many connected components.

For the heat equation on~$\real^d\times \real$, there are no stable compacts $K\ne\emptyset$
(compacts~$K$ are thin at the points $(x,t)\in K$ with minimal $t$). 

However, this does not 
mean that our assumptions only hold in  the elliptic case. Consider, for example, the closed upper
half space $X:=\{x\in \real^d\colon x_d\ge 0\}$, $d\ge 2$, with Brownian motion on $X^+:=\{x\in \reald \colon x_d>0\}$
which continues as Brownian motion on the hyperplane $H:=X\setminus X^+$ when exiting from $X^+$.
By symmetry, it is easy to write down explicitly the corresponding harmonic kernels for the half balls 
\begin{equation*} 
B_r^+(x):=\{ y\in X\colon |y-x|<r\}, \qquad x\in H, \, r>0,
\end{equation*} 
using the Poisson integral for balls in $\real^d$ and $\real^{d-1}$. The situation is not elliptic, since, for example,
  $H_{B_1^+(0)}(0,\cdot)$ is supported by $\{z\in H\colon |z|=1\}$.

From now on,  let us assume in this section that we have a normal Bauer family $(H_V)_{V\in \V}$ of harmonic boundary 
kernels on $X$ which is elliptic. 
The   proofs of Lemma~\ref{point},  Proposition \ref{LW-stable} and Corollary  \ref{K-U-regular},(1) 
are strongly influenced by \cite[Chapitre I,7]{Herve}.

\begin{lemma}\label{point}
Let $x\in X$ and $V$ be an open neighborhood of $x$. Then there exists 
a~strongly stable compact neighborhood $K$ of $x$ in $V$.
\end{lemma}

\begin{proof} We may assume without loss of generality that 
$V$ is a connected regular set and there exists 
a~strict potential $p\in\C_0(V)$. Let $\vp\in \C(V)$ with compact support~$L$
 in~$V$ such that $\vp(x)>0$. Then $q:=\vp\circ p\in\C_0(V)$, $q>0$,
and $q$ is harmonic on~$V\setminus L$. Let $0<\a<\inf q(L)$ and 
  define $K:=\{x\in V\colon q(x)>\a\}$, $h:=1_{V\setminus K}(\a-q)$.
\end{proof}

\begin{proposition}\label{LW-stable}
Let $L $ be a compact in a~$\mathcal P$-set $W$.
Then there exists a~stable compact neighborhood $K$ of $L$ in $W$. 
\end{proposition} 

\begin{proof} By Lemma \ref{point}, for every $x\in L$,  there exists a~stable compact neighborhood~$L_x$
of $x$ in $W$. There are $x_1,\dots,x_n\in L$ such that $L$ is covered by 
the open sets $\overset\circ L_j$, $1\le j\le m$. 
The proof is completed   taking  $ K:=L_{x_1}\cup\dots \cup L_{x_m}$.
\end{proof} 

\begin{corollary}\label{K-U-regular}
For every compact $K$ in a $\mathcal P$-set $W$, there exists a 
regular neighborhood  of $K$ in $W$.
\end{corollary}

\begin{proof}  Let $W'$ be an open neighborhood of $K$ with $\ov W'\subset W$.
  By Proposition~\ref{LW-stable}, there exists a stable compact neighborhood $L'$ of $\partial W'$ in $W\setminus K$.
  Obviously, $W'\setminus L'$ is a regular  neighborhood of $K$ in $W$.
\end{proof} 

\begin{proposition}\label{components}
 Let $K$ be a stable compact in a $\mathcal P$-set $W$. Then $K$ is strongly stable
if and only if $K$ has at most finitely many holes, that is, there are only finitely 
many connected components $V$ of $W\setminus K$ which are relatively compact in $W$
{\rm(}or -- equivalently -- have a compact closure in $X$ with $\partial V\subset K${\rm)}.
\end{proposition} 

\begin{proof} Suppose first that $K$ has at most finitely many holes
and let $W'$ be obtained from $W$ omitting one point from each hole.
 By Corollary \ref{K-U-regular}, there is  a~regular neighborhood~$U$ of~$K$ in~$W'$. 
Then $\partial U$ intersects every connected component of~\hbox{$X\setminus K$}.
 Hence $h:=H_{U\setminus K}1$ is strictly positive on $U\setminus K$ (and tends to zero at $\partial K$).

The converse follows by the same argument as in the proof of Lemma \ref{finite}.
\end{proof} 

\bibliographystyle{plain}


\begin{thebibliography}{1}


\bibitem{bauer-silov}
H.~Bauer.
\newblock \v {S}ilovscher {R}and und {D}irichletsches {P}roblem.
\newblock {\em Ann. Inst. Fourier Grenoble}, 11:89--136, XIV, 1961.

\bibitem{Bauer66}
H. Bauer.
\newblock {Harmonische {R}\"aume und ihre {P}otentialtheorie}.
\newblock {\em Lecture Notes in Mathematics 22.}
\newblock Springer, Berlin 1966.


\bibitem{BH}
J.~Bliedtner and W.~Hansen.
\newblock {\em {Potential Theory -- An Analytic and Probabilistic Approach to
  Balayage}}.
\newblock Universitext. Springer, Berlin, 1986.

\bibitem{chen}
Z.-Q.~Chen.
\newblock{\em Topics on recent developments in the theory of Markov processes.}
Lectures at RIMS, Kyoto University, 2012.

\bibitem{chen-fuku-book}
Z.-Q.~Chen, M.~Fukushima. 
\newblock{\em Symmetric Markov Processes, Time Change, and Boundary Theory}.
\newblock Princeton University Press, 2012.


\bibitem{chen-fuku}
Z.-Q.~Chen and  M.~Fukushima.
\newblock Stochastic  Komatu-Loewner equation and Brownian motion with darning in 
multiply connected domains.
\newblock{\em Preprint 2013}.
\newblock{To appear in Trans. Amer. Math. Soc.}

\bibitem{chen-fuku-rohde}
Z.-Q.~Chen, M.~Fukushima and S.~Rhode.
\newblock Chordal Komatu-Loewner  evolutions and BMD domain constant.
\newblock{\em arXiv:1410.8257v1}.

\bibitem{chen-lou}
Z.-Q.~Chen and S.~Lou.
\newblock{Brownian motion on spaces with varying dimension}.
 \newblock{In preparation}.

\bibitem{Const}
C.~Constantinescu and A.~Cornea.
\newblock {\em {Potential Theory on Harmonic Spaces}}.
\newblock {Grundlehren d. math. Wiss.} Springer, Berlin - Heidelberg - New
  York, 1972.

\bibitem{GH1}
A.~Grigor'yan and W.~Hansen.
\newblock {A Liouville property for Schr\"{o}dinger operators}.
\newblock {\em Math. Ann.}, 312:659--716, 1998.

\bibitem{H-pert}
W.~Hansen.
\newblock Perturbation of harmonic spaces and construction of semigroups.
\newblock {\em Invent. Math.}, 19:149--164, 1973.

\bibitem{H-strict}
W.~Hansen.
\newblock Some remarks on strict potentials.
\newblock {\em Math. Z.}, 147:279--285N, 1976.

\bibitem{H-coupling}
W.~Hansen.
\newblock Modification of balayage spaces by transitions with applications
to coupling of PDE's.
\newblock {\em Nagoya Math. J.}, 169: 77--118, 2003.

\bibitem{Herve}
{R.-M.} Her{v\'e}.
\newblock Recherches axiomatiques sur la th\'eorie des fonctions surharmoniques
  et du potentiel.
\newblock {\em Ann. Inst. Fourier}, 12:415--517, 1962.
\end{thebibliography}

{\small \noindent 
Wolfhard Hansen,
Fakult\"at f\"ur Mathematik,
Universit\"at Bielefeld,
33501 Bielefeld, Germany, e-mail:
 hansen$@$math.uni-bielefeld.de}

\end{document}